\documentclass[12pt]{amsart}
\usepackage{amsmath, amssymb}
\usepackage{graphicx}
\usepackage{url}

\newtheorem{theorem}{Theorem}
\newtheorem{conjecture}[theorem]{Conjecture}
\newtheorem{observation}[theorem]{Observation}
\newtheorem{corollary}[theorem]{Corollary}

\newtheorem{open}[theorem]{Open Problem}
\newtheorem{definition}{Definition}

\newtheorem{remark}[]{Remark}
\theoremstyle{definition}

\def\bth{\begin{theorem}}
\def\eth{\end{theorem}}
\def\bc{\begin{corollary}}
\def\ec{\end{corollary}}
\def\bcj{\begin{conjecture}}
\def\ecj{\end{conjecture}}

\title[]{Chooser-Picker Degree Games for Regular Graphs}

\author[Gy\H{o}rffy]{Lajos Gy\H{o}rffy }\address{Bolyai Institute,
University of Szeged and John von Neumann University, Kecskemét}

\email{lgyorffy@math.u-szeged.hu}


\begin{document}

\begin{abstract}
In the unbiased Chooser-Picker (also known as Client-Waiter) game played on the edge set of a graph, Picker offers a pair of unclaimed edges in each turn, Chooser claims one, and the remaining edge goes back to Picker. We study the Chooser-Picker (C-P) degree game played on $d$-regular graphs, where Chooser aims to maximize the maximum degree of their induced subgraph, and Picker's objective is to defend every vertex by securing a certain minimum degree in Picker's own subgraph. While classical static pairing strategies guarantee a minimum degree of at least $\lfloor d/4 \rfloor$ for Breaker on general $d$-regular graphs in Maker-Breaker (M-B) games and for Picker in C-P games, outperforming this threshold has been a major open challenge in both frameworks. According to the foundational monograph of J. Beck, this challenge stands as the first among the seven most humiliating problems in combinatorial game theory.
Our main result is that Picker can beat the $d/4$ bound. First, we prove that Picker can always guarantee a degree of at least one at every vertex on any $3$-regular graph. Based upon this we introduce a direct strategy to prove that Picker can secure a degree of at least $\lfloor d/3 \rfloor$ at every vertex for any $d$-regular graph. This highlights a fundamental structural advantage that Picker usually possesses over Breaker in sparse local games.
\end{abstract}

\maketitle

\section{Introduction} \label{intro}

The mathematical analysis of positional games on graphs, pioneered by Beck \cite{BB} and extensively developed under the Maker-Breaker framework, has experienced significant growth over the past few decades (e.g., \cite{BP, BMP, Beck2, Hefetz}). Among the various formulations, the unbiased Chooser-Picker (C-P) and Picker-Chooser (P-C) games (see, \cite{BB, Bed1, Bed3, Cser, CMP, CMP2})---originally introduced by J. Beck and frequently referred to as Client-Waiter and Waiter-Client games in modern literature, (\cite{Bed2, CGH, Dvo, Dvo2, HKT1, HKT2, Kriv1, Tan})---present a compelling dynamic of asymmetric control.

In each turn of a standard C-P (or P-C) game played on the edge set of a base graph $G$, Picker selects a pair of previously unclaimed edges and offers them to Chooser. Chooser must claim exactly one of these edges for their own subgraph, while the remaining edge is automatically allocated to Picker. This process continues sequentially until the entire edge set of the base board is exhausted.

This paper focuses specifically on the \emph{degree game} variant mostly within this C-P framework. In a C-P degree game, Chooser's objective mimics that of Maker in a standard Maker-Breaker game, aiming to maximize the maximum degree of their induced subgraph. Conversely, Picker acts as the defending player, analogous to Breaker, whose goal is to defend the graph globally by ensuring that Picker secures a high minimum degree in their own subgraph at the end of the game. (The roles are swapped in the P-C framework.) To formalize Picker's performance in the C-P degree game, for any simple graph $G$, similarly to $\delta_B(G)$ introduced in \cite{cube} for M-B games for Breaker, we define $\delta_P(G)$ to be the maximum minimum degree that Picker can guarantee in their own subgraph at the end of the game.

On dense boards, such as the complete graph $K_n$, Picker can leverage global potential function arguments to achieve a nearly equitable split of $n/2 - o(n)$ edges per vertex, see Krivelevich et al.\ \cite{Kriv2, Kriv1}. However, on fixed sparse $d$-regular graphs, local topological constraints severely restrict Picker's options. 
For decades, the standard benchmark for any $(1:1)$ degree game on a $d$-regular graph has been the $\lfloor d/4 \rfloor$ bound. This baseline is established via a classical static pairing strategy derived from an Eulerian orientation of $G$, forcing Chooser to award Picker at least $\lfloor d/4 \rfloor$ edges at every single vertex.

Breaking the $\lfloor d/4 \rfloor$ bound for sparse regular graphs has stood as a long-standing open problem in combinatorial game theory. Indeed, for the closely related Maker-Breaker framework, J. Beck \cite{BB} listed this challenge as the first among the \emph{seven most humiliating open problems} of positional game theory. While static graph decompositions and global factorization methods can sometimes be utilized to outperform this threshold on specific symmetric structures (as explored in our companion paper \cite{cube}), such techniques often fail to translate to general graphs.

The main contribution of this paper is to demonstrate that Picker can utilize the sequential, dynamic nature of the offering mechanism to universally break through this $\lfloor d/4 \rfloor$ threshold on \emph{any} regular graph. Our first core result establishes the strategic foundation on the sparsest non-trivial regular graph family:

\begin{theorem} \label{CP3}
	In the Chooser-Picker degree game played on an arbitrary $3$-regular graph $G$, Picker has a strategy to secure a degree of at least one at every vertex.
\end{theorem}

By building upon the algorithmic intuition of this cubic base case, we introduce a direct offering mechanism that naturally scales to higher degrees. Via this direct sequential strategy, we establish our main universal result:

\begin{theorem} \label{CPd}
	In the C-P degree game played on an arbitrary $d$-regular graph $G$, Picker has a strategy to secure a degree of at least $\lfloor d/3 \rfloor$ at every single vertex.
\end{theorem}

Theorem~\ref{CPd} highlights a fundamental structural advantage that Picker possesses over Breaker in sparse local games. 
To demonstrate the robust nature of this dynamic framework, we also extend our generalized offering mechanism to alternative positional game settings. Specifically, we investigate a variant of the Walker-Breaker degree game (introduced by Espig et al.~\cite{Esp}, studied in \cite{Clem, Forc1}) and establish a similar $\lfloor d/3 \rfloor$ defensive threshold for Breaker. Furthermore, we introduce the \emph{Connector-Breaker game}, similarly to London and Pluh\'ar~\cite{Lond}, where Connector (as Maker) can achieve one degree in the 3-regular graphs, but the strategy fails on $d$-regular graphs aiming $\lfloor d/3 \rfloor$ edges at every vertex. A brief exploration of the dual Picker-Chooser (Waiter-Client) configurations shows that there for Chooser (as Breaker) even the standard baseline of $\lfloor d/4 \rfloor$ is too hard, Picker may get all edges incident to a vertex in a $4$-regular graph.

Related positional frameworks have also been explored on sparse graphs recently. For instance, Forcan and Mikala\v{c}ki~\cite{Forc2} analyzed a Maker-Breaker total domination framework where players select vertices rather than edges to control dominating structures on graphs. In contrast, the ``Toucher and Isolator'' setup proposed by Dowden et al.~\cite{TI} relies exclusively on edge allocation, shifting the focus to enumerating how many vertices Isolator can completely isolate against optimal defense.

The paper is structured as follows. In Section \ref{prel}, we provide the necessary preliminaries, formal definitions, and briefly cross-reference the relevant baselines established in \cite{cube}. Section~\ref{results} formally introduces our main results, encompassing the $3$-regular case, the general $d$-regular strategy, and the extensions to Walker-Breaker, Connector-Breaker and Picker-Chooser games. Section~\ref{proofs} details the formal proofs of these theorems. Finally, Section~\ref{furt} offers concluding remarks and highlights several open problems.

\section{Preliminaries and Definitions} \label{prel}

\subsection{Origins}

Positional games on hypergraphs were formalized by Hales and Jewett \cite{HJ}, who provided the first general results in the field. Later, Erd\H{o}s and Selfridge \cite{erdos} introduced a numerical criterion for Breaker's victory, establishing the seminal potential-function method for Maker-Breaker games. While these pioneering works primarily focused on hypergraph games, a profound connection between graph games and random graph theory was observed. The degree game variant, where players compete to optimize local vertex densities, was subsequently studied under Maker-Breaker rules by Beck \cite{BB, Beck5} and Sz\'ekely \cite{Szekely}. For the sake of completeness, we recall the standard definitions from our companion paper \cite{cube}:

\begin{definition}[Degree Game]
	Let $G$ be an arbitrary simple graph. In the $(1:1)$ Maker-Breaker degree game, Maker and Breaker alternately claim previously unclaimed edges of $G$, one at a time. Maker's objective is to maximize the maximum degree of their induced subgraph, while Breaker's objective is to minimize it.
\end{definition}

As noted in \cite{cube}, it is often mathematically advantageous to quantify Breaker's defensive target explicitly by shifting the perspective to the number of edges secured at each local neighborhood:

\begin{definition}[$t$-Degree Game]
	Let $G$ be an arbitrary simple graph. In the $t$-degree game, Maker and Breaker alternately claim previously unclaimed edges of $G$. Breaker's objective is to secure at least $t$ edges at every vertex. Conversely, Maker's objective is to claim at least $\deg_G(x) - t + 1$ edges at some vertex $x \in V(G)$, where $\deg_G(x)$ denotes the degree of $x$ in $G$.
\end{definition}

To measure Breaker's optimal defensive capabilities under this formulation, we employ the parameter $\delta_B(G)$, defined as the maximum minimum degree that Breaker can guarantee in their own subgraph at the end of the Maker-Breaker degree game. When the underlying board $G$ is restricted to the class of $d$-regular topologies, a well-known folklore pairing strategy yields a universal baseline defense for Breaker:

\begin{observation}[Folklore, cf. \cite{cube, BB}] \label{euler}
	Let $G$ be an arbitrary simple $d$-regular graph. Then $\delta_B(G) \ge \lfloor d/4 \rfloor$.
\end{observation}

This $\lfloor d/4 \rfloor$ lower bound is tightly connected to static graph properties. Indeed, by constructing an Eulerian orientation of a $d$-regular graph (or an appropriate auxiliary expansion if $d$ is odd), one can define a rigid, non-reactive pairing strategy that forces an equitable edge allocation. This deterministic guarantee stands in sharp contrast to the bounds obtained from static graph colorings, where a traditional discrepancy approach yields a significantly higher local equity:

\begin{observation}[Folklore, cf. \cite{cube}] \label{obs:color_discrepancy}
	The edges of any $d$-regular graph $G$ can be $2$-colored such that at every vertex $v \in V(G)$, the number of incident red and blue edges is $\lfloor d/2 \rfloor \pm 1$.
\end{observation}

The substantial gap between the static coloring bound of approximately $d/2$ and the game-theoretic pairing bound of $\lfloor d/4 \rfloor$ highlights the competitive advantage that the attacker possesses in dynamic selection processes. Overcoming this pairing barrier has stood as one of the most prominent challenges in the field. In his foundational monograph, Beck in his monograph \cite{BB}, conceptualized this difficulty by ranking the problem of outperforming the Eulerian pairing strategy as the first among the \emph{seven most humiliating open problems} in positional game theory.

\begin{open}[Beck \cite{BB}, cf. \cite{cube}] \label{hum}
	Can we improve the lower bound of $\lfloor d/4 \rfloor$ in the Maker-Breaker degree game to some $c \cdot d$, where $c > 1/4$?
\end{open}

\begin{conjecture} (Folklore) \label{nirvana}
	For any $d$-regular graph $G$, $\delta_B(G) \ge d/2 - o(d)$ holds.
\end{conjecture}

While our companion paper \cite{cube} focuses on breaking this $\lfloor d/4 \rfloor$ bound in Maker-Breaker games by leveraging the specific geometric and product symmetries of hypercubes, grids, and tori, the global structural restrictions of that framework leave the general case wide open. In the present paper, we shift our focus to the Chooser-Picker framework. We demonstrate that by harnessing the sequential, dynamic power of the offering mechanism, the defending player can universally improve the defense threshold from $\lfloor d/4 \rfloor$ to $\lfloor d/3 \rfloor$ on any $d$-regular graph.

\subsection{Chooser-Picker Games}
In his seminal formulation, Beck~\cite{Beck5} defined the \emph{Picker-Chooser} (P-C) and the \emph{Chooser-Picker} (C-P) versions of positional games (referring to the famous cake-cutting algorithm), that has achieved remarkable prominence in combinatorial game theory. This paradigm has been extensively investigated and developed in the last few decades.
Here the two competing entities are designated as \emph{Picker} and \emph{Chooser}. In each round, Picker selects an unselected pair of elements from the board and presents them to Chooser. Chooser retains exactly one of these elements for their own collection and returns the remaining element to Picker. If the total number of elements on the board is odd, the final remaining solitary element is automatically assigned to Chooser. In the C-P version, Chooser acts as the attacker (Maker) while Picker plays as the defender (Breaker); conversely, these adversarial roles are completely swapped in the P-C framework.

Beck observed an intriguing phenomenon in positional game theory: a winning state for Breaker in a standard M-B game and a winning state for Picker in the corresponding C-P variant appear to occur concurrently across several natural board structures. However, there are some counterexamples to this connection, see Knox~\cite{Knox}.

\begin{remark}
	We note that Chooser-Picker and Picker-Chooser games are frequently formalized in modern literature as Client-Waiter and Waiter-Client games, respectively, where Chooser takes the role of the Client and Picker acts as the Waiter. However, throughout this paper, we strictly adhere to the original terminology established by the inventor, J. Beck, out of respect for historical precedence.
\end{remark}

The operational rules for the Chooser-Picker (C-P) and Picker-Chooser (P-C) \emph{Degree Games} and \emph{$t$-Degree Games} adapt these mechanics directly to the edge set $E(G)$ of a simple graph $G$. In each turn, Picker offers a pair of previously unoccupied edges to Chooser, who claims one for $G_C$, while the other is awarded to Picker's subgraph $G_P$. In the C-P configuration, Picker plays defensively as Breaker, whereas in the P-C setup, Picker acts as Maker.

To quantify Picker's defensive strength under the primary C-P framework, we introduce a parameter analogous to the classical Maker-Breaker setting:

\begin{definition}
	For any simple graph $G$, $\delta_P(G)$ denotes the maximum minimum degree that Picker can guarantee considering exclusively the edges in $G_P$ at the end of the Chooser-Picker degree game.
\end{definition}

Naturally, when the board $G$ is a $d$-regular graph, a universal baseline defense of $\lfloor d/4 \rfloor$ applies to $\delta_P(G)$ as well. To achieve this, Picker can construct a static, non-reactive strategy by offering pairs of edges that correspond directly to the outgoing edges from each vertex in an Eulerian orientation of $G$, mirroring the classic folklore pairing utilized in Maker-Breaker games.

\subsection{Walker-Breaker and Connector-Breaker games}

To evaluate the power of our dynamic offering principles, we investigate two alternative Maker-Breaker configurations where Maker's strategic freedom is subjected to strict local or global topological constraints. 
The first variant under consideration is the \emph{Walker-Breaker} (W-B) game, which belongs to a class of constrained positional games originally introduced by Espig et al.\ \cite{Esp} and extensively analyzed in \cite{Clem, Forc1}. In this setting, the attacking player, designated as \emph{Walker}, takes the role of Maker but is subjected to a severe local movement restriction:

\begin{definition}[Walker-Breaker Game]
	Let $G$ be a simple graph. The Walker-Breaker (W-B) degree game is played on the edge set $E(G)$. Walker (acting as Maker) opens the game by claiming an initial edge $e_1$. In each turn, Walker is strictly constrained to claim an unoccupied edge $e_{i+1}$ that is incident to the terminal endpoint of their previously claimed edge $e_i$, constructing a continuous walk on $G$. Conversely, Breaker may claim any available edge on the entire board during their turn.	
	(If, at any point, Walker reaches a vertex from which no unclaimed incident edges remain, Walker is permitted to start a new walk by selecting an arbitrary unoccupied edge.)
\end{definition}

Under this framework, we prove a strong defensive theorem for Breaker, demonstrating that the exact same $\lfloor d/3 \rfloor$ defense threshold holds for W-B games as well. Specifically, Breaker can secure a minimum degree of at least $\lfloor d/3 \rfloor$ at every vertex when playing on a $d$-regular graph $G$.

The second constrained framework evaluated in this paper is the \emph{Connector-Breaker} game, which imposes a global structural constraint on the attacker's subgraph rather than a local movement rule:

\begin{definition}[Connector-Breaker Game]
	In the Connector-Breaker variant of the Maker-Breaker degree game, Connector (acting as Maker) can only take unoccupied edges that keep their graph connected, if possible. (In case of no unoccupied edges connected to their graph, Connector can choose an arbitrary edge.)
\end{definition}

While we show that this global connectivity constraint allows Breaker to guarantee a minimum degree of at least one on any $3$-regular graph, we note that this property does not scale analogously to higher degrees, establishing a sharp contrast with the universal applicability of the Chooser-Picker strategy.

\section{Results} \label{results}

In this section, we present the structural and game-theoretic bounds established across different frameworks on regular and general graphs. We begin by stating our primary results under the classical Chooser-Picker setup, showing that the dynamic, sequential offering mechanism allows Picker to universally improve the static bound of $\lfloor d/4 \rfloor$.

\subsection{Chooser-Picker Games}

Our first result focuses on $3$-regular graphs, establishing the strategic foundation for the general defense mechanism.

\begin{theorem}[Theorem~\ref{CP3} reformulated]
	For any $3$-regular graph $G$, we get $\delta_P(G) \ge 1.$
\end{theorem}

We note that under the classical M-B framework, the analogue of this statement does not hold for Breaker. Indeed, Breaker cannot guarantee a positive minimum degree for several well-known cubic graphs. However, in some specific $3$-regular cases, Breaker is capable of securing at least one edge at every vertex, with prominent examples including $Q_3, K_4$, and $K_{3,3}$ (as detailed in \cite{cube}).

The next theorem is our main result. By generalizing the dynamic sequential offering mechanism developed for the cubic base case, we scale this defensive advantage to arbitrary $d$-regular graphs. 

\begin{theorem}[Theorem~\ref{CPd} reformulated]
	Let $G$ be $d$-regular. Then $\delta_P(G) \ge \left\lfloor d/3 \right\rfloor.$
\end{theorem}

A remarkable consequence of our generalized offering framework is that its strategic execution does not depend on the global regular structure of the board. Instead, the strategy operates purely on a local, node-by-node reactive basis. This property allows us to lift Theorem~\ref{CPd} to any simple graph, yielding a strong localized degree theorem:

\begin{theorem} \label{CPgen}
	Let $G = (V,E)$ be an arbitrary simple graph (not necessarily regular). In the Chooser-Picker degree game, Picker has a winning strategy to secure a degree of at least $\lfloor \deg_G(v)/3 \rfloor$ at every single vertex $v \in V$ in Picker's induced subgraph $G_P$.
\end{theorem}

Theorem~\ref{CPgen} highlights a fundamental competitive advantage that Picker possesses over Breaker in sparse local games, as a comparable universal localized result remains an elusive open challenge under the rigid, classical Maker-Breaker constraints.

Next, we evaluate the robustness of our sequential offering principles by applying them to the constrained attacker frameworks detailed in Section 2. We observe that when Maker's global or local freedom is restricted, Breaker can establish an identically powerful defense.

\subsection{Walker-Breaker and Connector-Breaker Games}
Our first constrained result demonstrates a strategic parity between the sequential dynamics of the Chooser-Picker setup and the localized step-constraints of the Walker-Breaker game:

\begin{theorem} \label{WBd}
	For any $d$-regular graph $G$, Breaker has a winning strategy in the Walker-Breaker degree game to guarantee a final degree of at least $\lfloor d/3 \rfloor$ at every vertex in Breaker's induced subgraph; that is, 
	$\deg_B(v) \ge \left\lfloor \frac{d}{3} \right\rfloor \quad \forall v \in V(G).$
\end{theorem}

\begin{theorem} \label{WBgen}
	Let $G$ be an arbitrary simple graph (not necessarily regular). In the Walker-Breaker degree game, Breaker has a winning strategy to secure a degree of at least $\lfloor \deg_G(v)/3 \rfloor$ at every single vertex $v \in V$ in Breaker's induced subgraph.
\end{theorem}

On the other hand, restricting the attacker via a global connectivity rule yields a specialized defense on low-order regular graphs that does not scale analogously to higher degrees of regular graphs:

\begin{observation} \label{conn}
	In the Connector-Breaker degree game played on any $3$-regular graph $G$, Breaker has a winning strategy to ensure that no vertex remains isolated in Breaker's subgraph, yielding $\deg_B(v) \ge 1$ for all $v \in V(G)$.
\end{observation}

\subsection{Picker-Chooser Games}
Finally, we turn our attention to the dual Picker-Chooser framework. As noted in Section~\ref{prel}, the scoring inversion completely transforms the underlying strategic layout: Picker now acts as Maker attempting to minimize the minimum degree in Chooser's induced subgraph. We characterize a stark topological contrast in this framework between different families of sparse graphs. 

\begin{observation} \label{PCP}
	Picker (acting as Maker) has a winning strategy to isolate at least one vertex in Chooser's subgraph when the Picker-Chooser game is played on the $3$ or $4$-dimensional hypercube graph $Q_3$ and $Q_4$ or on the Petersen graph.
\end{observation}

As we can see, even the classical $\lfloor d/4 \rfloor$ Eulerian pairing strategy does not carry over to Picker-Chooser configurations, as Picker can overcome this threshold on the $4$-dimensional (and thus, $4$-regular) hypercube.

\begin{observation} \label{PCC}
	Chooser (acting as Breaker) has a winning strategy to defend all vertices, ensuring a minimum degree of at least one in Chooser's subgraph, when the game is played on $K_4$ or on $K_{3,3}$.
\end{observation}

Following this comprehensive overview of the threshold bounds, all relevant proofs and explicit sequential strategies are presented in detail in Section \ref{proofs}.

\section{Proofs} \label{proofs}
\subsection{Proof of the Chooser-Picker 3-Regular Base Case}

We first provide the formal proof of Theorem~\ref{CP3} by presenting an explicit sequential offering strategy for Picker on an arbitrary $3$-regular graph $G$. To formalize the dynamic behavior of this strategy, we introduce the geometric concept of a \emph{cherry}.

\begin{definition}[Cherry]
	A cherry consists of a pair of adjacent edges sharing a common vertex. This common vertex is called the \emph{center} of the cherry, while the two remaining distinct vertices of degree one are referred to as the \emph{endpoints} of the cherry.
\end{definition}

Picker's global objective is to ensure that by the end of the game, every vertex is incident to at least one edge in Picker's induced subgraph $G_P$. During the game, we classify the status of the vertices based on the edges claimed so far. A vertex $v \in V(G)$ is defined as \emph{safe} if it is already incident to at least one edge in $G_P$. Conversely, a vertex $v$ is defined as \emph{dangerous} if it is incident to exactly one edge in Chooser's subgraph $G_C$ and is not incident to any edges in $G_P$. A dangerous vertex poses an immediate structural threat because if Chooser manages to claim all of its remaining incident edges, Picker will fail to defend it.
Picker's winning strategy is executed via a dynamic, reactive cherry-offering mechanism as follows:

\noindent\textbf{Initial Move:} Picker selects an arbitrary vertex $v_0 \in V(G)$ and offers a cherry centered at $v_0$, consisting of two of its incident unclaimed edges. Chooser claims one of these edges for $G_C$, while the other edge is allocated to $G_P$. Consequently, the endpoints of Picker's edge become immediately safe. However, the endpoint of Chooser's edge, say $v_1$, now has one incident edge in $G_C$ and none in $G_P$. Since $G$ is $3$-regular, $v_1$ is left with exactly two unoccupied edges, making it the unique dangerous vertex on the board.

\noindent\textbf{Subsequent Reactive Moves:} In each turn, if a unique dangerous vertex $v$ exists, Picker is forced to offer the cherry centered precisely at $v$, utilizing its remaining two unoccupied incident edges. Following this offer, Chooser must claim one edge for $G_C$, leaving the other for $G_P$. This yields two crucial structural consequences:
\begin{enumerate}
	\item The previous dangerous vertex $v$ is successfully resolved and becomes safe, as it now receives exactly one edge in $G_P$ (alongside two edges in $G_C$).
	\item Chooser's new edge extends into a new vertex, say $w$. If $w$ was previously unvisited (and thus had a local degree of zero in both subgraphs), it now receives its first incident edge in $G_C$. Since $w$ has two remaining unoccupied edges, it becomes the new unique dangerous vertex for the next round.
\end{enumerate}

We note that at the end of any such turn, at most one new dangerous vertex can emerge. This is because Chooser's choice is restricted to the two edges offered within the cherry, and the endpoint of Picker's edge is automatically rendered safe.

If Chooser's new edge terminates at a vertex $w$ that is already safe (meaning it already possesses at least one incident edge in $G_P$ from an earlier stage of the game), then no dangerous vertices remain on the board. In this scenario, Picker gains strategic freedom and can initiate the next round by offering an arbitrary cherry centered at any remaining vertex. This process continues until all vertices are safe.

To highlight the necessity of this localized strategy, we observe that offering independent edges is highly suboptimal. If Picker were to offer two independent edges spanning four distinct vertices, Chooser's choice would simultaneously create two new dangerous vertices (the two endpoints of Chooser's claimed edge). If Chooser could subsequently claim a path connecting these two dangerous points, they could successfully isolate a neighborhood and win. By strictly offering adjacent edges in the form of a cherry, Picker guarantees that at most one dangerous vertex exists at any given stage, which can be immediately neutralized in the subsequent turn. This completes the proof of Theorem~\ref{CPd}. \qed

\begin{figure}[htbp]
	\centering
	\includegraphics[scale=0.46]{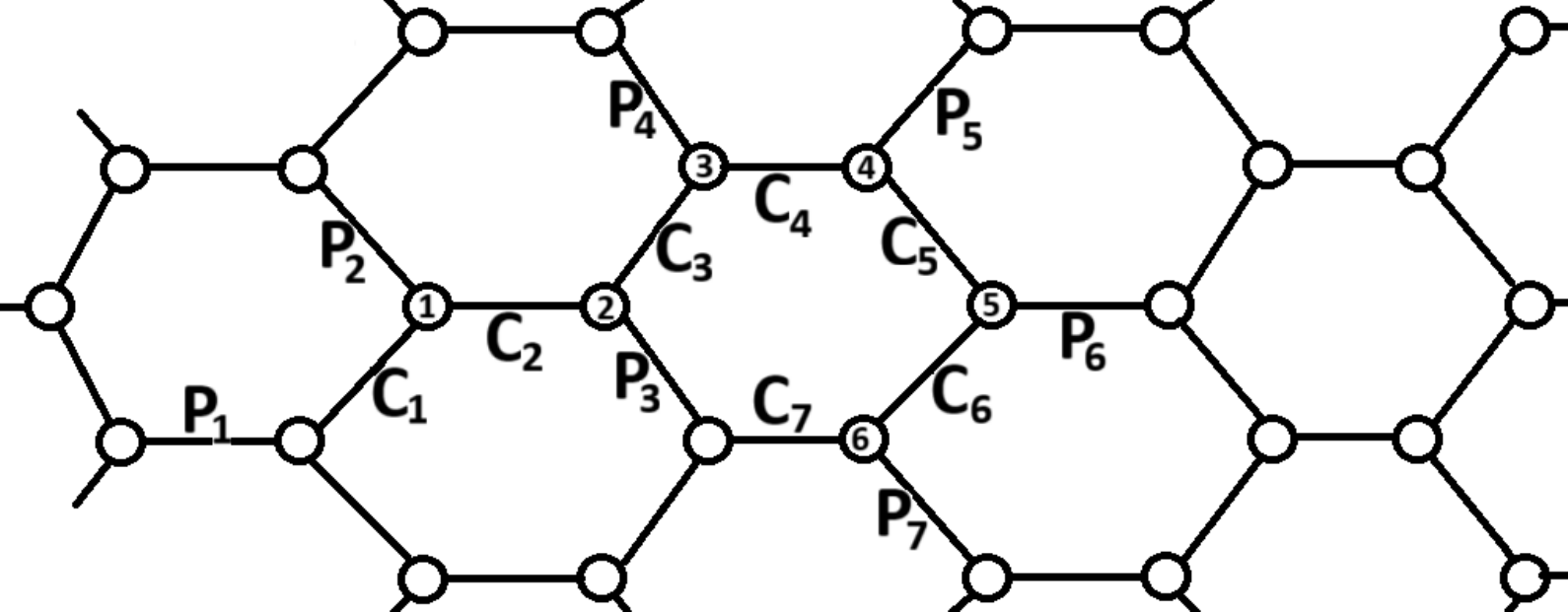}
	\caption {A Chooser-Picker game}
	\label{3re}
\end{figure}

Figure~\ref{3re} illustrates an illustrative example of a C-P game executed via the dynamic offering strategy described above. Here, $P_i$ and $C_i$ denote the $i$-th move of Picker and Chooser, while the numbers labeled inside the vertices indicate the location of the unique dangerous vertex directly after the $i$-th turn. 
Note that following the $7$-th move in the diagram, Chooser's claimed edge terminates at an already safe vertex. Consequently, no dangerous vertices remain after this step, allowing Picker to freely initiate a new offering phase by presenting a cherry at an entirely new location. As a structural consequence of this mechanism, the edges claimed by Chooser form a continuous walk that originates at their initial selection and terminates at their final edge right before Picker initiates a new phase of cherries from an unvisited vertex.

\subsection{Proof of the Chooser-Picker General Degree Theorem}
We now present the proof of Theorem~\ref{CPd} and its universal extension, Theorem~\ref{CPgen}. The proof establishes that the localized, reactive cherry-offering mechanism developed for the cubic base case can be scaled directly to any graph.

\begin{proof}[Proof of Theorem~\ref{CPd} and Theorem~\ref{CPgen}]
	Let $G = (V, E)$ be an arbitrary simple graph, where the local degree of any vertex $v \in V$ is denoted by $\deg_G(v)$. We describe a generalized, sequential cherry-offering strategy for Picker that ensures Picker secures at least $\lfloor \deg_G(v)/3 \rfloor$ edges incident to every vertex $v$.
	
	The game is executed via successive phases of the reactive cherry-offering walk described in Section 4.1. Picker maintains the global count of active edges. A vertex $v$ is designated as \emph{dangerous} if it has been entered by a Chooser-claimed edge but has not yet been neutralized by a Picker-claimed edge during the current local round. The game proceeds as follows:
	\begin{enumerate}
		\item Picker initiates a phase by selecting an arbitrary vertex $v_0$ that possesses at least two adjacent unoccupied edges, offering a cherry centered at $v_0$.
		\item As long as a dangerous vertex $v$ exists and possesses at least two remaining unoccupied incident edges, Picker immediately reacts by offering a cherry randomly, centered precisely at $v$.
		\item If the walk terminates because Chooser's new edge lands on a vertex $w$ that has no two unoccupied edges, Picker regains strategic freedom and can initiate a new phase from any available vertex with a degree of at least two.
	\end{enumerate}

	Let us analyze the local edge accounting at an arbitrary vertex $v \in V$ during this process. Each time the sequential offering mechanism visits and transitions through $v$ (i.e., when $v$ acts as the center of an offered cherry), the local neighborhood of $v$ loses exactly three unoccupied edges from the remaining board:
	Exactly one edge is claimed by Chooser to enter the vertex $v$ in the preceding step. And exactly two edges are offered by Picker as a cherry centered at $v$, of which Chooser claims one and Picker receives the other.
	Consequently, in every complete local transition where $v$ acts as the center of the cherry, the available degree of $v$ decreases by exactly three, out of which Chooser accumulates exactly two edges, and Picker secures exactly one edge. 
	
	We note that during the game, Picker may also acquire edges incident to $v$ where $v$ acts merely as an endpoint of a cherry centered at some neighboring vertex. These additional edges do not negatively impact the strategic lower bound; instead, they strictly increase the final degree in Picker's induced subgraph $G_P$.
	
	Therefore, if $v$ serves as the center of an offered cherry for a total of $m_v$ times, Chooser gets at most $2m_v$ edges. Because a vertex can only cease to be a valid cherry center when its remaining unoccupied degree drops below two, and each complete transition consumes three incident edges, we get the following: the total number of complete transitions that can be forced through $v$ is at most $\lfloor \deg_G(v)/3 \rfloor$. 
	
	Even if Chooser claims all of the remaining leftover edges at the end of the game when no further cherries can be centered at $v$, Picker's required quota is already safely accumulated. At the end of the game, the final degree secured by Picker at any vertex $v \in V$ satisfies:
	\[ \deg_{G_P}(v) \ge \left\lfloor \frac{\deg_G(v)}{3} \right\rfloor. \]
	If $G$ is a $d$-regular graph, $\deg_G(v) = d$ for all $v \in V$, which yields $\delta_P(G) \ge \lfloor d/3 \rfloor$, simultaneously completing the proofs of Theorem~\ref{CPd} and Theorem~\ref{CPgen}.
\end{proof}

We conclude this section by showing that the dynamic cherry-offering framework is inherently bounded by this threshold, regardless of how Picker optimizes the choice of the offered cherry from the current dangerous center.

\begin{observation} \label{CPtight}
	Using the algorithm described above, the guaranteed minimum degree secured by Picker cannot be improved by employing any deterministic rule for cherry selection from the dangerous vertex instead of choosing an arbitrary one.
\end{observation}

\begin{proof}
	Chooser can designate an untouched vertex after the initial move as their \emph{special vertex}. Whenever Picker offers a cherry containing an edge incident to this special vertex, Chooser always claims that edge. According to Picker's strategy, their subsequent offer must be a cherry centered precisely at this special vertex, where Chooser claims one of the two edges, thereby accumulating at least two-thirds of the local edges.
\end{proof}

\subsection{Proof for Walker-Breaker Games}

We now provide the generalized proof of Theorem~\ref{WBd} and Theorem~\ref{WBgen}. By building upon the identical triplet-depletion mechanism developed for the Chooser-Picker framework, we establish that the defensive threshold scales directly to any neighborhood without global regularity constraints.

\begin{proof}[Proof of Theorem~\ref{WBd} and Theorem~\ref{WBgen}]
	Let $G = (V,E)$ be an arbitrary simple graph. We present a strictly reactive local response strategy for Breaker. Let Walker's sequential continuous walk be tracked by the sequence of vertices it visits: $v_1, v_2, v_3, \dots$. 
	
	Breaker's strategy is defined as follows: whenever Walker transitions through a vertex $v_i$ by entering via an edge $\{v_{i-1}, v_i\}$ and immediately departing via $\{v_i, v_{i+1}\}$, Breaker responds during their subsequent turn by claiming an arbitrary remaining unoccupied edge incident to this intermediate vertex $v_i$. If Walker gets trapped in a dead end (a vertex with no remaining unoccupied incident edges) and initiates a new walk component from an arbitrary vertex $u_1$ by claiming an edge $\{u_1, u_2\}$, Breaker plays as in case of $v_i$ sequence.
	
	The local edge accounting at any vertex $v \in V(G)$ follows an identical paradigm to the proof of Theorem~\ref{CPgen}. Each time Walker's walk transitions through $v$, the local neighborhood of $v$ loses exactly three unoccupied edges: two are claimed by Walker (one entering and one departing) and one is claimed reactively by Breaker. 
	
	Analogous to the Chooser-Picker triplet-depletion process, each complete local transition reduces the available degree of $v$ by exactly three, with Walker acquiring two edges and Breaker securing one edge. This cycle can be forced at $v$ for at most $\lfloor \deg_G(v)/3 \rfloor$ times before the remaining free degree at $v$ drops below $3$. The endpoints of Breaker's edges outside the Walker's walk only strictly increase the final degree in Breaker's induced subgraph $G_B$. 
	
	Thus, at the end of the game, Breaker's final degree satisfies $\deg_{G_B}(v) \ge \lfloor \deg_G(v)/3 \rfloor$ at every single vertex $v \in V$. If $G$ is $d$-regular, $\deg_G(v) = d$ for all $v \in V$, yielding $\deg_{G_B}(v) \ge \lfloor d/3 \rfloor$, which simultaneously completes the proofs of Theorem~\ref{WBd} and Theorem~\ref{WBgen}.
\end{proof}

\subsection{Proof for Connector-Breaker Games}

We provide a concise verification of Observation~\ref{conn}, evaluating how global connectivity constraints affect the local degree thresholds.

\begin{proof}[Proof of Observation~\ref{conn}]
	The proof employs a reactive strategy for Breaker that maintains an identical invariant to the $3$-regular Chooser-Picker base case: at the start of any turn, there exists at most one dangerous vertex on the board. 
	
	Connector opens the game by claiming an initial edge $e_1 = \{v_1, v_2\}$. Breaker immediately claims an adjacent edge $\{v_1, w\}$, leaving $v_2$ as the unique dangerous vertex. Because Connector is bound by a strict global constraint to keep their induced subgraph connected at all times, they cannot claim independent edges or connect two distinct dangerous points (as another one never exists). Furthermore, it is strictly suboptimal for Connector to touch an already safe vertex. Thus, Connector is forced to extend their structure from the unique dangerous vertex to an unvisited vertex, moving exactly like Chooser along the cherry-offering walk described in Section 4.1. Breaker reactively claims the third and final edge of the dangerous vertex, rendering it safe and transferring the vulnerability to the newly reached vertex. By induction, this chain ensures that $\deg_B(v) \ge 1$ holds for all $v \in V(G)$.
\end{proof}

	To conclude, we demonstrate why this defense fails to scale to higher degrees ($d > 3$). When $d = 3$, any visited vertex is immediately exhausted (receiving 2 edges for Connector and 1 for Breaker). However, if $d > 3$, the vertices along Connector's existing path still possess multiple unoccupied incident edges. In this higher-degree regime, Connector can claim an unoccupied internal edge between two previously visited vertices. Because this single move modifies the degree count at both endpoints simultaneously, it shifts the local balance to an unfavorable $3:1$ ratio. Breaker can only neutralize one threat (reducing it to $3:2$), leaving the other at a $3:1$ state. In the next round, Connector can repeat this internal shortcutting maneuver, creating two new $3:1$ points. Facing three accumulating threats, Breaker's single available move is insufficient. By expanding along a continuous path, Connector sequentially drives the ratio to $4:1$ and eventually to an unblockable $5:1$ ratio on the next turn, completely breaking the $\lfloor d/3 \rfloor$ defensive threshold.

\subsection{Proofs for Picker-Chooser Games}

We now verify the strategic thresholds for the P-C framework stated in Observations~\ref{PCP} and \ref{PCC}. Recall that in this framework, Picker acts as Maker whose objective is to force an isolated vertex in Chooser's subgraph (or equivalently, to claim all incident edges at some vertex in $G_P$).

\begin{proof}[Proof of Observation~\ref{PCP}, Case $Q_3$]
	Picker can force an isolated vertex in Chooser's subgraph on $Q_3$ via an explicit four-step sequential offering strategy. 
	
	In the first step, Picker offers a pair of opposite parallel edges belonging to the same dimension. Chooser claims one, and the remaining edge is awarded to Picker. In the second step, Picker offers the remaining two opposite parallel edges in the exact same directional orientation. Regardless of Chooser's choices, Picker successfully secures two parallel edges on a single two-dimensional face of the cube. 
	
	In the third step, Picker offers the remaining two edges belonging to that specific face. Following Chooser's response, Picker accumulates three edges on that face, which creates two distinct vertices, say $A$ and $B$, each incident to exactly two Picker-claimed edges and one remaining unoccupied edge. In the fourth and final step, Picker offers these two remaining third edges incident to $A$ and $B$ as a pair. Chooser is forced to select one, leaving the other to Picker. The edge awarded to Picker necessarily completes the local neighborhood of either $A$ or $B$, ensuring that all three incident edges at that vertex belong to $G_P$. This forces an isolated vertex in Chooser's subgraph, completing the proof for $Q_3$.
\end{proof}

Following the same paradigm of localized containment, the proof of Picker's win on the Petersen graph (where Picker secures two incident edges at a vertex on the outer cycle in the first two steps, then replicates this on the inner cycle at a vertex non-adjacent to the first, and finally concludes by offering the two remaining unclaimed edges at these degree-two vertices to secure the final required edge), as well as Chooser's defense on $K_4$ and $K_{3,3}$, are straightforward but tedious finite game trees, which we leave to the interested reader.

\begin{proof}[Proof of Observation~\ref{PCP}, Case $Q_4$]
	We extend the hypercube isolation strategy to the $4$-dimensional hypercube $Q_4$. Let us view $Q_4$ (which is $4$-regular) as a perfect matching on two disjoint $Q_3$ graphs: $Q^{(1)}_3$ and $Q^{(2)}_3$.
	
	Picker initiates the strategy by executing the first three steps of the $Q_3$-isolation strategy simultaneously and independently on both subcubes. Consequently, Picker secures three edges on a designated face within \emph{both} $Q^{(1)}_3$ and $Q^{(2)}_3$. At this stage of the game, Picker possesses two vertices in the first subcube (labeled $A$ and $B$) and two vertices in the second subcube (labeled $C$ and $D$) such that each of these four vertices is incident to exactly two Picker-claimed edges within their respective subcubes, leaving two unoccupied edges incident to each in $G$. 	
	The structural analysis branches into three cases:
	\begin{enumerate}
		\item[(i):] The vertex pairs $\{A, B\}$ and $\{C, D\}$ share zero cross-cube edges. In this scenario, Picker offers the two remaining internal subcube edges incident to $A$ and $B$ as a pair, and subsequently offers the two remaining internal subcube edges incident to $C$ and $D$ as another pair. Following these rounds, Picker achieves a degree of three at two distinct vertices across the subcubes. Picker then finishes the game by offering the pair of their remaining fourth cross-cube edges, forcing all four incident edges of one vertex into $G_P$.
		
		\item[(ii):] The vertex pairs share exactly one cross-cube edge. Without loss of generality, let $A$ and $C$ be adjacent via a cross-cube edge, while $B$ and $D$ are non-adjacent. Picker first offers the remaining internal subcube edges incident to the non-adjacent vertices $B$ and $D$. Chooser claims one, forcing Picker to achieve a degree of three at the other (say, at vertex $B$). Next, Picker offers the internal subcube edges incident to $A$ and $C$, forcing Picker to secure a degree of three at one of them (say, at vertex $C$). Finally, Picker offers the pair consisting of the fourth edge of $B$ and the fourth edge of $C$. Chooser's forced choice guarantees that Picker achieves all four incident edges at either $B$ or $C$.
		
		\item[(iii):] The vertex pairs share exactly two cross-cube edges, meaning $\{A, C\}$ and $\{B, D\}$ are both edges in $G$. Picker opens this phase by offering these two cross-cube edges $\{A, C\}$ and $\{B, D\}$ as a pair. Picker receives one of them, say $\{A, C\}$. This move immediately raises the degree count of both $A$ and $C$ in $G_P$ to three. Picker then successfully concludes the game by offering the pair consisting of the final remaining fourth edges incident to $A$ and $C$, ensuring a full degree of four at one of these vertices.
	\end{enumerate}

	This explicit construction for $Q_4$ demonstrates that in the Picker-Chooser framework, the defensive capabilities are significantly weaker, showing that even the classical $\lfloor d/4 \rfloor$ lower bound cannot be universally guaranteed for Chooser.
\end{proof}

\section{Further Remarks and Open Questions} \label{furt}

As demonstrated throughout this paper, in the context of \emph{Chooser-Picker} degree games, we have successfully established a universal lower bound that strictly outperforms the classical $\lfloor d/4 \rfloor$ baseline. This result provides a definitive affirmative answer to the analogue of the first humiliating question posed by J. Beck, confirming that the defense barrier can be shifted to $\lfloor d/3 \rfloor$ under the C-P rules.

Analogous dynamic offering frameworks can be effectively adapted to \emph{Walker-Breaker} games. The structural concepts of constrained attackers, such as Walker-type or Connector-type players, have received substantial attention in recent literature. 

In \emph{Picker-Chooser} games, the underlying strategic dynamics operate under fundamentally different principles than those in the C-P framework. Nonetheless, we have characterized a sharp topological contrast on sparse environments: we identified specific low-order regular graphs where Picker can successfully force an isolation (such as $Q_3$ and even $Q_4$) and only two where Chooser wins (such as $K_4$ and $K_{3,3}$).

Naturally, the most prominent and historically challenging positional setups remain the classical \emph{Maker-Breaker} games, from which these foundational problems originate. While the grand conjecture states that the true strategic equilibrium satisfies $\delta_B(G) = (1/2 - o(1))d$, even the relaxed bound $\delta_B(G) \ge (1/4 + \epsilon)d$ remains wide open for arbitrary graphs. In our companion paper \cite{cube}, we demonstrated that for specific structural classes of $d$-regular boards (such as hypercubes and toroidal grids), the $\lfloor d/4 \rfloor$ bound can indeed be overcome to yield $\delta_B(G) \ge \lfloor d/3 \rfloor$.

To further evaluate the boundary between Breaker's defensive limitations and strategic advantages under the Maker-Breaker rules, it is instructive to consider the framework of biased games. While outperforming the $\lfloor d/4 \rfloor$ baseline remains extremely difficult in the unbiased setting, Breaker's task becomes straightforwardly manageable even under a minimal move advantage, as demonstrated by the following local immediate-response property:

\begin{observation} \label{obs:biased_mb}
	In the $(1:2)$ biased Maker-Breaker degree game played on an arbitrary $d$-regular graph $G$, Breaker has a winning strategy to secure a minimum degree of at least $\lfloor d/2 \rfloor$ at every vertex.
	
	Moreover, in the general $(1:2k)$ biased Maker-Breaker degree game, Breaker can guarantee a minimum degree of at least $\lfloor \frac{k}{k+1}d \rfloor$ at every single vertex.
\end{observation}

\begin{proof}
	This follows from a direct, localized immediate-response strategy for Breaker. Whenever Maker claims an edge $e = \{u, v\} \in E(G)$, Breaker responds by claiming exactly one available edge incident to $u$ and one incident to $v$ in the $(1:2)$ game. In the general $(1:2k)$ game, Breaker allocates exactly $k$ moves to claim available edges incident to $u$ and $k$ incident to $v$. 
	Consequently, Breaker successfully secures the required fraction of the total degree at every vertex, completing the proof.
\end{proof}

To stimulate future research along these intersecting boundaries of local degree security, we submit the following open problems to the community:

\begin{open}
	Can Picker achieve a minimum degree larger than $\lfloor d/3 \rfloor$ in the standard Chooser-Picker degree game played on an arbitrary $d$-regular graph $G$?
\end{open}

\begin{open}
	Can Picker (acting as Maker) win the $1$-degree game on $Q_d$ for any $d>4$?
\end{open}

\begin{open}
	What is the value of $\delta_B(G)$ in the classical M-B degree game played on general $d$-regular graphs? In particular, is it true that Breaker can achieve $\lfloor d/3 \rfloor$ lower bound that we established for Picker in the Chooser-Picker framework?
\end{open}

\section*{Acknowledgments}
The author would like to express his gratitude to Andr\'as Pluh\'ar and Andr\'as London. Their insightful comments, valuable suggestions, and constant encouragement were indispensable throughout the writing of this paper.

Declaration on the use of generative AI. Gemini was used in a limited manner during manuscript preparation for language editing and organizational assistance. 


\end{document}